\documentclass[letterpaper,11pt]{article}

\usepackage{graphicx,color,epsfig}
\usepackage{amsfonts,amsmath,amssymb,amsthm}

\newtheorem{thm}{Theorem}[section]
\newtheorem{prop}[thm]{Proposition}
\newtheorem{lem}[thm]{Lemma}
\newtheorem{lemA}{Lemma}[section] 							
\newtheorem{cor}[thm]{Corollary}

\newtheorem{asm}{Assumption}
\newtheorem{op}[thm]{Question}

\theoremstyle{remark}

\theoremstyle{definition}
\newtheorem{defn}{Definition}

\newcommand{\ra}{\rightarrow}

\newcommand{\N}{\mathbb N}     
\newcommand{\R}{\mathbb R}     
\newcommand{\Z}{\mathbb Z}     

\renewcommand{\d}{\delta}

\newcommand{\e}{\varepsilon}
\newcommand{\del}{\partial}
\newcommand{\w}{\omega}

\renewcommand{\l}{\lambda}
\renewcommand{\L}{\Lambda}
\newcommand{\bL}{\overline{\Lambda}}
\newcommand{\s}{\sigma}

\renewcommand{\P}{\mathbb{P}}   
\newcommand{\bP}{\overline{\mathbb{P}}}
\newcommand{\E}{\mathbb{E}}   
\newcommand{\bE}{\overline{\mathbb{E}}}

\newcommand{\vp}{\mathrm{v}_P}
\newcommand{\iid}{i.i.d.\ }

\title{On the Annealed Large Deviation Rate Function for a Multi-Dimensional Random Walk in Random Environment}
\author{Jonathon Peterson\thanks{
School of Mathematics, 
University of Wisconsin, 480 Lincoln Drive, Madison, WI 53705. 
The research of the author was partially supported by NSF grants 
DMS-0503775 and DMS-0802942 and by a Doctoral Dissertation 
Fellowship from the University of Minnesota.} \and Ofer Zeitouni\thanks{School
of Mathematics, University of Minnesota, 206 Church St. SE, Minneapolis, 
MN 55455 and Faculty of Mathematics, Weizmann Institute of Science, Rehovot
76100, Israel.
The research of the author was partially supported by NSF grants  
DMS-0503775 and DMS-0804133,
and by a grant from the Israel Science Foundation.}
}
\date{December 14, 2008. Revised July 9, 2009}

\begin{document}

\maketitle

\begin{abstract}
We derive properties of the rate function in Varadhan's (annealed) large
deviation principle for multi-dimensional, ballistic
random walk in random environment, in a certain neighborhood 
of the zero set of the rate function. Our approach relates
the LDP to that of regeneration times and distances. The analysis 
of the latter is possible due to the i.i.d. structure of regenerations.
\end{abstract}

\begin{section}{Introduction and Statement of Main Results}\label{RWRELDP}

This paper studies
annealed large deviations
for multi-dimensional random walks in random environments (RWRE),
in the ballistic regime. 
We will be concerned with nearest neighbor 
RWRE with uniformly elliptic \iid environments, modeled as follows. 
Let $\mathcal{E}_d:= \{ x\in \Z^d: \|x\| = 1\}$ and let $\Omega := \left( \mathcal{M}(\mathcal{E}_d) \right)^{\Z^d}$, where $\mathcal{M}(\mathcal{E}_d)$ is the space of all probability measures on $\mathcal{E}_d$.
Let $\mathcal{F}$ be the $\s$-field generated by the cylinder sets
of $\Omega$, and let $P$ be a probability measure on
$(\Omega, \mathcal{F})$. A {\it random environment}
$\w = \{ \w(x, \cdot) \}_{x\in \Z^d}$ is an 
$\Omega$-valued random variable with distribution $P$. 
Given an environment $\w$, the {\it quenched}
law $P_\w$ of a RWRE $X_n$ starting at the origin $\mathbf{0}$
is given by
\[
 P_\w( X_0 = \mathbf{0} ) = 1 \quad\text{and}\quad
 P_\w \left( X_{n+1} = x + y | X_n = x \right) = \w(x,y). 
\]
The {\it annealed}
(also called the averaged) law $\P$ of a 
RWRE starting at the origin is defined by 
\[
 \P( \cdot ) = \int P_\w( \cdot ) P(d\w). 
\]
Expectations with respect to the measures $P_\w$ and $\P$ will be 
denoted by $E_\w$ and $\E$, respectively. 

For the remainder of this paper we will assume that the law on environments 
is \emph{uniformly elliptic} and 
\emph{i.i.d.} That is, we will make the following assumptions:
\begin{asm}[Uniformly Elliptic]\label{asmUE}
 There exists an $\e>0$ such that $P\left(\w(0,x) \geq \e, 
 \: \forall x \in \mathcal{E}_d \right)  = 1$.
\end{asm}
\begin{asm}[\iid environments]\label{asmIID}
 The law on environments $P$ is an \iid product measure. That is, 
 $\{ \w(x, \cdot) \}_{x\in \Z^d}$ are \iid under $P$. 
\end{asm}

Recall the following classification of laws on the environment 
that was first introduced in \cite{zLDP}.
\begin{defn}
Let $d(\w):= E_\w X_1$ be the drift at the origin of the environment, 
and let $\mathcal{K}$ be the closure of the convex hull of the support,
under $P$, of all possible drifts. 
If $\mathbf{0} \in \mathcal{K}$, then $P$ is \textbf{nestling}. 
If $P$ is nestling but 
$\mathbf{0}$ does not belong to the 
 interior of $\mathcal{K}$, then $P$ is \textbf{marginally nestling}.
If $\mathbf{0}\notin \mathcal{K}$, then $P$ is \textbf{non-nestling}. 
If $\ell \in \R^d \backslash \{\mathbf{0} \}$ and
$\inf_{x \in \mathcal{K}} x \cdot \ell > 0$ then $P$ is
\textbf{non-nestling in direction $\ell$}. 
\end{defn}

Varadhan has proved the following annealed large deviation principle (LDP)
for RWRE.
\begin{thm}[Varadhan \cite{vLDP}]\label{annealedLDP}
Let Assumptions \ref{asmUE} and \ref{asmIID} hold. 
Then, there exists a convex function $H(v)$ such that $\frac{X_n}{n}$ 
satisfies an annealed large deviation principle with good 
rate function $H(v)$. That is, for any Borel subset $\Gamma \subset \R^d$, 
with $\Gamma^\circ$ denoting its interior  and
$\overline{\Gamma}$ its closure,
\[
-\inf_{v\in{\Gamma^\circ} } H(v) \leq \liminf_{n\ra\infty} 
\frac{1}{n} \log \P \left( \frac{X_n}{n} \in \Gamma \right) 
\leq \limsup_{n\ra\infty} \frac{1}{n} \log \P\left( \frac{X_n}{n}
\in \Gamma \right) \leq -\inf_{v\in\overline{\Gamma}} H(v).
\]
Moreover, the zero set of the rate function 
$Z:= \{ v: H(v)=0\}$ is a single point if $P$ is non-nestling and 
a line segment containing the origin if $P$ is nestling.
\end{thm}
\noindent\textbf{Remark:} 
As shown in \cite{vLDP},
the conclusion of Theorem \ref{annealedLDP} holds 
more generally
for RWRE with bounded jumps in i.i.d. environments with 
certain strong uniform ellipticity conditions. 

Until recently, other than this description of the zero set of $H$ 
and the fact that $H$ is convex, no other qualitative properties of
the annealed rate function were known. 
In contrast, much more is known about the qualitative behavior 
of the annealed rate function when $d=1$. 
In \cite{cgzLDP}, a rather detailed description of
the one-dimensional annealed rate function was given. 
In particular, intervals were identified on which the annealed rate 
function is strictly convex, and sufficient conditions 
were given for the annealed rate function to have linear pieces
in a neighborhood of the origin.

Recently, Yilmaz \cite{yThesis}, \cite{yPaper}
has made some progress on the understanding
of the annealed rate function for multi-dimensional RWRE and on
the distribution of paths leading to large deviations. 
He has shown that under certain conditions 
on the environment, there exist regions where the annealed 
rate function is strictly convex and analytic. 
In this paper, we provide a different
proof of these results, and also provide further 
information about the annealed rate function when 
the environment is nestling. 
In particular, we show in the latter case
that there is an open set which has the origin 
in its boundary and on which the annealed 
rate function is analytic and $1$-homogeneous 
(that is, $H(c v) = c H(v)$ if $v$ and $cv$ are both in the open set). 

Our approach to analyzing the annealed large deviations of
multi-dimensional RWRE utilizes what are known as \emph{regeneration times}.
Recall that for $\ell \in S^{d-1}:= \{ \xi \in \R^{d}: 
\| \xi \| = 1 \}$ such that $c\ell \in \Z^d$ for some $c>0$, 
regeneration times in the direction $\ell$ may be defined by
\[
\tau_1 := \inf \{ n > 0 : X_k \cdot \ell < X_n 
\cdot \ell \leq X_m \cdot \ell , \quad \forall k<n, \quad \forall m\geq n \},
\]
and
\[
\tau_i := \inf \{ n > \tau_{i-1} : X_k \cdot \ell <
X_n \cdot \ell \leq X_m \cdot \ell , 
\quad \forall k<n, \quad \forall m\geq n \}, \quad\mbox{for } i > 1.
\]
Our final assumption is what is known as Sznitman's condition \textbf{T}.
To introduce it, define the event {\it escape to $+\infty$ in direction
$\ell$}: $A_\ell := 
\{ \lim_{n\ra\infty} X_n \cdot \ell = + \infty \}$.
\begin{asm}[Condition \textbf{T}]\label{asmT}
 Let $\ell \in S^{d-1}$ be such that 
 $c \ell \in \Z^d$ for some $c>0$, and such that the following hold.
 Either $P$ 
is non-nestling in direction $\ell$,  or $P$ is nestling and
\begin{itemize}
 \item[(i) ] $\P( A_\ell) = 1$. 
 \item[(ii) ] There exists a constant $C_1>0$ such that 
\[
 \E \exp \left\{ C_1 \sup_{0\leq n\leq \tau_1} \| X_n \| \right\} < \infty,
\]
where $\tau_1$ is the first regeneration time in direction $\ell$. 
\end{itemize}
\end{asm}
\noindent
\textbf{Remarks:} \textbf{1.} When $P$ is non-nestling in direction
$\ell$, 
it is straightforward to check using results of 
Sznitman \cite{sSlowdown} that (i) and (ii) above hold. See
Section \ref{nnbJprop} for more information.\\ 
\textbf{2.}
We require $c \ell \in \Z^d$ for some $c>0$ 
in order to allow for a simpler definition of regeneration times
that agrees with the one given by Sznitman and Zerner \cite{szLLN} 
(set $a = \frac{1}{c}$ in the definition of regeneration times 
in \cite{szLLN}). This restriction is not essential, as
\cite[Theorem 2.2]{sConditionT} implies that Assumption \ref{asmT} is
equivalent to the version of condition \textbf{T} given
in \cite{sConditionT} which does not require that $c \ell \in \Z^d$.

When $P$ is non-nestling or $d\geq 2$, 
Assumptions \ref{asmUE}, \ref{asmIID} and \ref{asmT} imply a 
law of large numbers with non-zero limiting velocity (see \cite{sConditionT}). 
That is, there exists a point $\vp\in \R^d\backslash 
\{ \mathbf{0}\}$ such that
\begin{equation}\label{LLN}
\lim_{n\ra\infty} \frac{X_n}{n} =: \vp, \quad \P-a.s. 
\end{equation}
Varadhan's description of the zero set of the annealed rate function in 
Theorem \cite{vLDP} implies that under 
Assumptions \ref{asmUE}, \ref{asmIID} and \ref{asmT},
if $P$ is non-nestling, 
then $\vp$ is the unique zero of the annealed rate function, and if 
$P$ is nestling then $[0, \vp]$ is the zero set of the annealed rate function.
Our main results are the following:
\begin{thm}\label{Thesis_LDPdiff}
Let Assumptions \ref{asmUE} and \ref{asmIID} hold and let $P$ be non-nestling. 
Then, the annealed rate function $H(v)$ is analytic and 
strictly convex in a neighborhood $\mathcal{A}'$ of $\vp$. 
\end{thm}

\begin{thm}\label{nestLDPdiff}
Let Assumptions \ref{asmUE}, \ref{asmIID}, and \ref{asmT} hold, 
let $P$ be nestling, and let $d\geq 2$. Then, there exists an open set 
$\mathcal{A}$ with the following properties:
\begin{enumerate}
 \item The half open interval $(0, \vp] \subset \mathcal{A}$. 
 \item $\mathcal{A}$ can be written as the disjoint union $\mathcal{A}
	 = \mathcal{A}^+ \cup \mathcal{A}^0 \cup \mathcal{A}^-$, where
	 $\mathcal{A}^+$ is open, 
	 $\mathcal{A}^0 \subset \del \mathcal{A}^+$ is a $d-1$ dimensional set
	 with non-empty (relative)
	 interior containing $\vp$, and $\mathcal{A}^- =
	 \{ c v: c\in(0,1), v \in \mathcal{A}^0 \}$. 
 \item The annealed rate function $H(v)$
	 is strictly convex and analytic on $\mathcal{A}^+$. 
 \item The annealed rate function $H(v)$ 
	 is analytic and $1$-homogeneous on $\mathcal{A}^-$.
 \item The annealed rate function $H(v)$ is 
	 continuously differentiable on $\mathcal{A}$.
\end{enumerate}
\end{thm}

\begin{figure}
\centering
\includegraphics{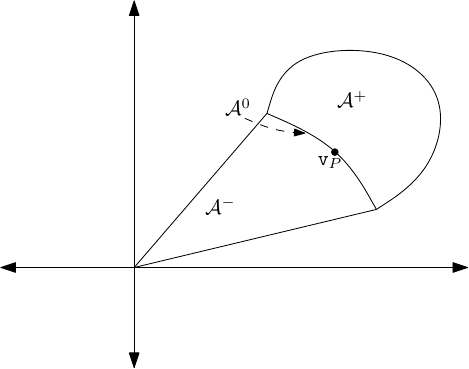}  
\caption{Theorem \ref{nestLDPdiff} describes the large deviation rate function on the set $\mathcal{A} = \mathcal{A}^+ \cup \mathcal{A}^0 \cup \mathcal{A}^-$ in the nestling case. The rate function is strictly convex and analytic on $\mathcal{A}^+$ and analytic and $1$-homogeneous on $\mathcal{A}^-$. The limiting velocity of the random walk $\vp \in \mathcal{A}^0$. }
\end{figure}

\noindent\textbf{Remarks:}  \textbf{1.}
Theorem 
\ref{Thesis_LDPdiff} was proved in \cite{MyThesis}. Independently,
Yilmaz proved in his thesis \cite{yThesis}
Theorem \ref{Thesis_LDPdiff} and part 3 of Theorem 
\ref{nestLDPdiff}. (Although stated under Kalikow's condition, which
is stronger than Assumption \ref{asmT}, his proof carries over
verbatim to the case where only Assumption \ref{asmT} holds.)
While Yilmaz's
proof also uses regeneration times, his approach is different in that
he does not introduce the rate function associated with
regeneration times and distances. (He also
depends on Theorem \ref{annealedLDP}, although this dependence can probably
be eliminated.) 
In contrast, our approach develops a
new proof of (a local version of) the
annealed large deviations,
independent
of Theorem \ref{annealedLDP}.
This alternative approach also allows one to make explicit an alternative
description of the rate function,
which in turn allowed us to deduce additional properties
of the annealed rate funtion\footnote{A. Yilmaz has kindly indicated to us 
how the representation in \cite{yPaper} can also 
be used to recover these additional properties; see
\cite[Proof of Theorem 4]{atarx} for details.}.\\
\textbf{2.} Theorem \ref{nestLDPdiff} still holds when $d=1$
(with ${\cal A}^0=\{ \vp \}$). 
This follows from the fact that $H(v) = 0$ for all $v \in [0,\vp]$ 
and the recent results of Yilmaz \cite{yThesis}, \cite{yPaper} which 
show that $H(v)$ is strictly convex and analytic on an open set bordering 
$\vp$. Our proof of Theorem \ref{nestLDPdiff} can be easily modified to also
cover the $d=1$ case, but for simplicity we will restrict ourselves in this 
paper to $d\geq 2$.

The structure of this paper is as follows: In Section \ref{regtimes}
we use the large deviations of regeneration times and distances
to define a new
function $\bar{J}(v)$. Most of Section \ref{regtimes} is devoted 
to proving qualitative properties of the function $\bar{J}(v)$. 
Section \ref{LDPlowerbound} provides an easy large deviation lower
bound with rate function $\bar{J}(v)$ for both nestling and non-nestling RWRE.
Then, in Section \ref{Thesis_upperbound} we derive matching large
deviation upper bounds in a neighborhood of $\vp$ when $P$ is
non-nestling 
and in the set $\mathcal{A}$ (which is defined in Section \ref{regtimes}) when $P$ is nestling. 
The proofs of Theorems \ref{Thesis_LDPdiff} and \ref{nestLDPdiff} are
then completed by noting that the large deviation upper and lower bounds 
proved in Sections \ref{LDPlowerbound} and \ref{Thesis_upperbound} imply 
that $\bar{J}(v) = H(v)$ for $v$ in appropriate subsets of
$\mathbb{R}^d$, and thus on these subsets,
the annealed rate
function $H(v)$ has the same properties 
that were proved for
$\bar{J}(v)$ in Section \ref{regtimes}. 
The Appendix contains a technical lemma on the analyticity of 
Legendre transforms that is used in Section \ref{regtimes}. 
\end{section}

\begin{section}{Regeneration Times and the 
	Rate Function $\bar{J}$}\label{regtimes}

	For the remainder of the paper, Assumptions \ref{asmUE},
	\ref{asmIID} and \ref{asmT} (with respect to a fixed direction
	$\ell$) will hold. 
Additionally, if $P$ is nestling, we will assume that $d\geq 2$. 
The regeneration times $\tau_i$ are obviously not stopping times
since they depend on the future of the random walk. They do however
introduce an i.i.d. structure, described next. Let 
$D:= \{ X_n \cdot \ell \geq 0, \; \forall n \geq 0 \}$. 
When $\P(D)>0$, let $\bP$ be the annealed law of the RWRE conditioned 
on the event $D$ (i.e., $\bP(\,\cdot \,) :=\P( \cdot \, | D )$). 
Expectations under the measure $\bP$ will be denoted  by $\bE$. 
\begin{thm}[Sznitman and Zerner \cite{szLLN}]\label{Thesis_regindep}
 Assume $\P(A_\ell) = 1$, and let $\tau_i$ be the regeneration times 
 in direction $\ell$. Then $\P(D)> 0$, and 
\[
 (X_{\tau_1}, \tau_1), (X_{\tau_2} -X_{\tau_1}, \tau_2 - \tau_1), 
 \ldots ,(X_{\tau_{k+1}} -X_{\tau_k}, \tau_{k+1} - \tau_k), \ldots
\]
are independent random variables. Moreover, the above sequence is 
i.i.d. under $\bP$.  
\end{thm}
\noindent\textbf{Remark:} 
The assumption that $\P(A_\ell)=1$ in Theorem \ref{Thesis_regindep} is 
only needed to ensure that $\tau_1 < \infty$. In fact, what is shown in 
\cite{szLLN} is that $\P(A_\ell ) > 0$ implies that $\P(D)>0$ and that 
$(X_{\tau_1}, \tau_1), (X_{\tau_2} -X_{\tau_1}, \tau_2 - \tau_1), \ldots$ 
are i.i.d. under $\bP$. 

Since Assumption \ref{asmT} requires that $\P(A_\ell)=1$, 
the conclusion of Theorem \ref{Thesis_regindep} is valid for 
the regeneration times $\tau_i$ in direction $\ell$. 
A consequence of Theorem \ref{Thesis_regindep} is the following 
useful formula for the limiting velocity $\vp$:
\begin{equation}
 \vp = \lim_{n\ra\infty} \frac{X_n}{n} = \frac{ \bE X_{\tau_1}}{\bE \tau_1}, 
 \qquad \P-a.s. \label{vPregenformula}
\end{equation}
Since either $P$ is non-nestling or $d\geq 2$ and
condition {\bf T} was assumed, 
it is known that 
$\bE \tau_1 < \infty $ and thus
$\vp \cdot \ell > 0$.
In fact, $P$ non-nestling or $d\geq 2$ 
imply that $\bE \tau_1^\gamma < \infty$ for 
all $\gamma < \infty$ (see \cite[Theorem 2.1]{sSlowdown} and \cite[Theorem 3.4]{sConditionT}).

Under the measure $\bP$,  $(X_{\tau_k}, \tau_k) = (X_{\tau_1}, \tau_1)
+ \sum_{i=2}^k (X_{\tau_i}-X_{\tau_{i-1}}, \tau_i-\tau_{i-1})$ is the
sum of \iid random variables. Therefore, a generalization of Cram\'er's
Theorem \cite[Theorem 6.1.3]{dzLDTA} implies that $\frac{1}{n}(X_{\tau_n},
\tau_n)$ satisfies a weak large deviation principle  under $\bP$ with
convex rate function 
\[
 \bar{I}(x,t):= \sup_{\eta\in \R^d, \; \l\in\R } \left[\left(\eta,\l\right) 
 \cdot \left(x,t\right) - \bL\left(\eta,\l\right)\right],
\]
where $\cdot$ denotes inner product and 
\[
\bL(\eta,\l):= \log \bE e^{(\eta,\l) \cdot (X_{\tau_1}, \tau_1)} 
\quad \text{ for } \eta\in\R^d, \; \l \in \R.
\]
In particular, for any open, convex subset $G\subset \R^{d+1}$,
\begin{equation}
 \lim_{k\ra\infty} \frac{1}{k} \log \bP\left( \frac{1}{k}(X_{\tau_k}, \tau_k) \in  G \right) = - \inf_{(x,t)\in G} \bar{I}(x,t). \label{Thesis_XtauLDPlb}
\end{equation}
Let $H_\ell := \{ v\in \R^d : v\cdot \ell > 0 \}$. Then, for $v\in H_\ell$, let
\[
\bar{J}(v):= \inf_{0<s\leq 1} s \bar{I}\left(\frac{v}{s},\frac{1}{s}\right).
\]
Our goal is to show that $\bar{J}(v) = H(v)$, at least for certain 
$v\in H_\ell$. The reasoning behind this is as follows.  
We assume that when $X_n \approx nv$ for some $v\in H_\ell$, 
the regeneration times occur in a somewhat regular manner
(that is, there are no extremely large regeneration times). If this
is the case, then it should be true that
\begin{equation} \label{LDapprox}
 \P(X_n \approx n v) \approx \bP(\tau_k \approx n,
 \; X_{\tau_k} \approx n v), \quad \text{ for } k=sn.
\end{equation}
However, the large deviations of $(X_{\tau_k}, \tau_k)/k$
imply that the latter probability is approximately $\exp\{ -n s
\bar{I}\left( \frac{v}{s}, \frac{1}{s} \right) \}$. The optimal $s$
for which \eqref{LDapprox} would hold must be the $s$ which minimizes
$s \bar{I}\left( \frac{v}{s}, \frac{1}{s} \right)$. 

The main difficulty in making the above heuristic argument precise 
comes in proving that there are no extremely long regeneration times 
when $X_n \approx n v$. 
In Section \ref{Thesis_upperbound} we resolve this difficulty for 
certain $v$ by showing that the least costly way to obtain a 
large deviation of $X_n \approx n v$ is to have all the regeneration 
times or distances relatively small. 

Having defined the function $\bar{J}$,
we now mention a few of basic properties. 
\begin{lem} \label{bJsimleprop}
$\bar{J}$ is a convex function on $H_\ell$, and $\bar{J}(\vp) = 0$. 
\end{lem}
\begin{proof}
For $s\in(0,1]$ and $v\in H_\ell$, let
\[
f(v,s):= s \bar{I}\left(\frac{v}{s},\frac{1}{s}\right) = \sup_{\eta \in \R^d, \;\l \in \R} (\eta, \l) \cdot( v, 1) - s \bL(\eta,\l). 
\]
Since $f(\cdot, \cdot)$ is the supremum of a family of 
linear functions, $f(\cdot, \cdot)$ is a convex function on
$H_\ell \times (0,1]$. 
Therefore, $\bar{J}(\cdot) = \inf_{s\in (0,1]} f(\cdot,s)$ is a
convex function on $H_\ell$. 

For the second part of the lemma, note that \eqref{vPregenformula} implies that $\bE X_{\tau_1} = \vp \bE \tau_1$.
Then, the law of large numbers and \eqref{Thesis_XtauLDPlb} imply that $\bar{I}(\vp \bE \tau_1, \bE \tau_1) = 0.$
The definition of $\bar{J}$ and the fact that $\bar{I}$ is non-negative imply that $\bar{J}(\vp)=0$. 
\end{proof}
We next evaluate some derivatives of $\bar{J}$.
For any function $g:\R^k\to (-\infty,\infty]$, let 
$\mathcal{D}_g= \{z\in \R^k: g(z)<\infty\}$ denote the {\it domain} of
$g$, and let
${\mathcal{D}^\circ_g}$ denote its interior. 
\begin{lem}\label{Jderivative}
Assume that $v_0$ and $s_0$ are such that $\bar{J}(v_0) = s_0 \bar{I}
 \left( \frac{v_0}{s_0}, \frac{1}{s_0} \right)$ and
 $(v_0/s_0, 1/s_0)\in {\mathcal{D}^\circ_{\bar{I}}}$.
Then,
\[
\frac{\del \bar{J}}{\del v_i} (v_0) = \frac{\del \bar{I}}{ \del x_i } 
\left( \frac{v_0}{s_0}, \frac{ 1}{s_0} \right) .
\]
\end{lem}
\begin{proof}
Since $(X_{\tau_1}, \tau_1)$ is, by Assumption
\ref{asmUE},
a non-degenerate $d+1$-dimensional random variable, 
$\bL(\eta, \l)$ is a strictly convex function on $\mathcal{D}_{\bar{\L}}$.
Since $\bar{I}$ is the Legendre transform of $\bL$, this implies that 
$\bar{I}(x,t)$ is continuously differentiable in 
${\mathcal{D}^\circ_{\bar{I}}}$ 
(see \cite[Theorem 26.3]{rConvex}). Therefore, 
$f(v,s)$ is continuously differentiable in the interior of
$\mathcal{D}_f =  \{ (v,s) : (v/s, 1/s) \in \mathcal{D}_I \}$.

Since $(v_0, s_0)\in
\mathcal{D}^\circ_f$, we have
$\frac{\del f}{\del s}(v_0, s_0) = 0$. 
Also, since $f(v,s)$ is convex as a function of $(v,s)$, 
\[
 f(v_0 + h e_i,s) \geq f(v_0, s_0) + \nabla f(v_0, s_0)
 \cdot( h e_i, s-s_0 ) = f(v_0, s_0) +  \frac{\del f}{\del v_i}(v_0,s_0) h,
\]
where in the second equality we used that
$\frac{\del f}{\del s}(v_0, s_0) = 0$. Since the right 
side of the above equation does not depend on $s$, we have
\begin{equation}\label{dJlb}
 \bar{J}(v_0 + h e_i) \geq f(v_0, s_0) +  
 \frac{\del f}{\del v_i}(v_0,s_0) h.
\end{equation}
On the other hand, a Taylor expansion of $f$ near $(v_0, s_0)$ implies that 
\begin{equation} \label{dJub}
 \bar{J}(v_0 + h e_i) \leq f(v_0 + h e_i, s_0) = 
 f(v_0, s_0) + \frac{\del f}{\del v_i}(v_0,s_0) h + o( h ). 
\end{equation}
Recalling that $\bar{J}(v_0) = f(v_0, s_0)$, \eqref{dJlb} and
\eqref{dJub} imply that $\frac{\del \bar{J}}{\del v_i} ( v_0, s_0) = 
\frac{\del f}{\del v_i}(v_0,s_0).$ The proof is completed by noting 
that the definition of $f(v,s)$ implies that 
$\frac{\del f}{\del v_i} (v,s) = \frac{\del \bar{I}}{\del x_i} (v/s, 1/s)$. 
\end{proof}

We now prove some more detailed properties of the function $\bar{J}(v)$ in the non-nestling and nestling cases, respectively. In particular, for certain $v$ we are able to identify the minimizing $s$ in the definition of $\bar{J}(v)$, and we are able to determine certain differentiability properties of $\bar{J}$. 

\begin{subsection}{Properties of $\bar{J}$ - Non-nestling Case} \label{nnbJprop}
When 
$P$ is non-nestling in direction $\ell$, 
the regeneration time $\tau_1$ has exponential tails 
\cite[Theorem 2.1]{sSlowdown}.
That is, Sznitman proved that there exists a constant $C_2>0$ such that 
\begin{equation} \label{tau1exptail}
 \bE e^{C_2 \tau_1} < \infty.
\end{equation}
Let
\[
\mathcal{C} := \{\eta\in \R^d: \| \eta \| < C_2/2 \} .
\]
If $\eta \in \mathcal{C}$, then $-C_2\tau_1/2  \leq \eta
\cdot X_{\tau_1} < C_2\tau_1/2 $ since $\|X_{\tau_1}\| \leq \tau_1$. Thus, 
\[
 1 = \bE e^{-C_2\tau_1/2  + C_2\tau_1/2 } < \bE e^{ \eta \cdot X_{\tau_1} +
 C_2\tau_1/2 } < \bE e^{C_2\tau_1/2  + C_2\tau_1/2 } < \infty, 
\]
and so  $\bL(\eta, C_2/2) \in (0, \infty)$ for all $\eta\in \mathcal{C}$.
Since $\bL(\eta, \l)$ is strictly increasing in $\l$ and since
$\lim_{\l\ra-\infty} \bL(\eta, \l) = -\infty$, we may define a
function $\l(\eta)$ on $\mathcal{C}$ by
\[
 \l(\eta) \text{ is the unique solution to } \bL(\eta, \l(\eta) )
 = 0, \quad \forall\eta\in \mathcal{C}. 
\]
Since $\bL$ is analytic in a neighborhood of $(\eta, \l(\eta))$ for any
$\eta \in \mathcal{C}$, a version of the implicit function 
theorem \cite[Theorem 7.6]{fgHolomorphic} implies that 
$\l(\eta)$ is analytic as a function of $\eta \in \mathcal{C}$. 
Differentiating the equality $\bL(\eta, \l(\eta)) = 0$, we obtain that
\[
 \nabla \l(\eta) = - \frac{ \bE X_{\tau_1} e^{\eta \cdot X_{\tau_1} + 
 \l(\eta) \tau_1} }{ \bE \tau_1 e^{\eta \cdot X_{\tau_1} + \l(\eta) \tau_1} }.
\]
This is useful in the proof of the following lemma.
\begin{lem} \label{nnanalytic}
Let $P$ be a non-nestling law on environments, and
let $\mathcal{A} := -\nabla \l( \mathcal{C} ) =
\{ - \nabla \l(\eta) : \eta \in \mathcal{C} \}$. Then,
$\vp \in \mathcal{A}$ and $\bar{J}$ is analytic and strictly
convex on the open set $\mathcal{A}$. 
Moreover, if 
\begin{equation} \label{vsnn}
 v_0 = -\nabla \l(\eta_0) = \frac{ \bE X_{\tau_1}
 e^{\eta_0 \cdot X_{\tau_1} + \l(\eta_0) \tau_1} }{ \bE \tau_1
 e^{\eta_0 \cdot X_{\tau_1} + \l(\eta_0) \tau_1} } 
 \qquad \text{and} \qquad s_0 = \frac{1}{\bE \tau_1
 e^{\eta_0 \cdot X_{\tau_1} + \l(\eta_0) \tau_1}} 
\end{equation}
for some $\eta_0 \in \mathcal{C}$, then 
\[
 \bar{J}(v_0) = s_0 \bar{I}\left( \frac{v_0}{s_0}, \frac{1}{s_0} \right), 
 \qquad \text{and} \qquad \nabla \bar{J}(v_0) = \eta_0,
\]
and $s_0$ is the unique value of $s$ which attains the
minimum in the definition of $\bar{J}(v_0)$.
\end{lem}
\begin{proof}
	Due to uniform ellipticity (Assumption \ref{asmUE}),
$\bL$ is strictly convex, and thus $\bL(\eta, \l(\eta)) = 0$ implies that
$\l(\eta)$ is strictly concave as a function of $\eta$. 
Therefore, $\nabla \l(\eta)$ is a one-to-one function on $\mathcal{C}$. 
Thus, $\mathcal{A}$ is an open set, 
and $\vp = \bE X_{\tau_1} / \bE \tau_1 = - \nabla \l( \mathbf{0}) 
\in \mathcal{A}^\circ$. 

Since $\bL$ is analytic and strictly convex in 
$\mathcal{D}^\circ_{\bar{\L}}$ and $\bar{I}$ is the Legendre transform of 
$\bL$, we have that $\bar{I}$ is analytic and strictly convex 
in the interior of $ \mathcal{D}'_{\bar{\L}} = \nabla 
\bL( \mathcal{D}_{\bar{\L}})$ (see Lemma \ref{Appendix_analytic} 
in Appendix \ref{Aanalytic}). Moreover, for any $(\eta,\l)\in
\mathcal{D}^\circ_{\bar{\L}}$,
\begin{equation} \label{LegendreDuals}
\bar{I}\left( \nabla \bL(\eta, \l)\right) = 
(\eta, \l) \cdot \nabla \bL(\eta, \l) - \bL(\eta, \l), 
\qquad \text{and} \qquad 
\nabla \bar{I} \left( \nabla \bL(\eta, \l) \right) = (\eta, \l).
\end{equation}

Letting $v_0$ and $s_0$ be defined as in \eqref{vsnn}, we have that 
$(v_0/s_0, 1/s_0 ) = \nabla \bL( \eta_0, \l(\eta_0))$. 
Recalling the definition of $f(v,s)$, we obtain that 
\[
 \frac{\del f}{\del s}(v,s) = \bar{I}\left( \frac{v}{s}, 
 \frac{1}{s} \right) - \nabla \bar{I}\left( \frac{v}{s}, 
 \frac{1}{s} \right) \cdot \left( \frac{v}{s}, \frac{1}{s} \right). 
\]
Therefore, 
\begin{align*}
 \frac{\del f}{\del s}(v_0,s_0) &= \bar{I}\left( \frac{v_0}{s_0}, \frac{1}{s_0} \right) - \nabla \bar{I}\left( \frac{v_0}{s_0}, \frac{1}{s_0} \right) \cdot \left( \frac{v_0}{s_0}, \frac{1}{s_0} \right)\\
&= \bar{I}\left( \nabla \bL(\eta_0, \l(\eta_0 )) \right) - \nabla \bar{I}\left( \nabla \bL(\eta_0, \l(\eta_0 )) \right) \cdot  \nabla \bL(\eta_0, \l(\eta_0 ))  \\
&= \bar{I}\left( \nabla \bL(\eta_0, \l(\eta_0 )) \right) - (\eta_0, \l(\eta_0) ) \cdot \nabla \bL(\eta_0, \l(\eta_0 ))  \\
&= -\bL(\eta_0, \l(\eta_0)) = 0, 
\end{align*}
where the third and fourth equalities follow from \eqref{LegendreDuals}.
Since $f(v,s)$ is convex as a function of $(v,s)$, 
it follows that $\bar{J}(v_0) = f(v_0, s_0) = s_0 \bar{I}(v_0/s_0, 1/s_0)$. 

Now,  with $D^2\bar{I}$ denoting the Hessian of $\bar I$,
\[
 \frac{\del^2 f}{\del s^2} (v,s) = \frac{1}{s^3} (v, 1) \cdot D^2 
 \bar{I}\left( \frac{v}{s}, \frac{1}{s} \right) (v,1)^t. 
\]
Since $\bar{I}(x,t)$ is strictly convex in a neighborhood of
$\nabla \bL(\eta_0, \l(\eta_0) ) = (v_0/s_0, 1/s_0)$, $D^2 \bar{I}(x,t)$ 
is strictly positive definite in a neighborhood of $(v_0/s_0, 1/s_0)$. 
Thus $\frac{\del^2 f}{\del s^2} (v_0,s_0) > 0$, 
and because $f(v,s)$ is analytic in a neighborhood of $(v_0, s_0)$, 
another use of  the implicit function theorem
\cite[Theorem 7.6]{fgHolomorphic} implies 
that 
there exists an analytic function $s(v)$ in a neighborhood of 
$v_0$ such that $s(v_0)= s_0$ and $\frac{\del f}{\del s} (v,s(v)) = 0$. 
Thus, $\bar{J}(v) = f(v, s(v))$, and therefore $\bar{J}(v)$ is 
analytic in a neighborhood of $v_0$. Moreover, since 
$\frac{\del^2 f}{\del s^2} (v_0,s_0) > 0$, $s_0$ is the unique 
value of $s$ obtaining the minimum in the definition of $\bar{J}(v_0)$. 

Since $\bar{J}(v) = f(v,s(v))$ in a neighborhood of $v_0$, 
$\bar{J}$ is strictly convex in a neighborhood of $v_0$ if $f(v,s)$ is strictly convex in a neighborhood of $(v_0, s_0)$. 
To see that $f(v,s)$ is strictly convex in a neighborhood of $(v_0, s_0)$, note that the definition of $f(v,s)$ implies that for $z\in\R^d$ and $w\in\R$, 
\[
 (z,w)^t \cdot D^2 f(v,s) \cdot (z,w) = \frac{1}{s} 
 \left( z - \frac{w}{s}v , \frac{-w}{s} \right) \cdot D^2 
 \bar{I}\left( \frac{v}{s}, \frac{1}{s} \right) \cdot 
 \left( z - \frac{w}{s}v , \frac{-w}{s} \right)^t.
\]
Since $D^2 \bar{I}(x,t)$ is strictly positive definite in a 
neighborhood of $(v_0/s_0, 1/s_0)$, this implies that $D^2 f(v,s)$ is 
strictly positive definite in a neighborhood of $(v_0, s_0)$, and 
thus $f(v,s)$ is strictly convex in a neighborhood of $(v_0, s_0)$. 

Finally, since $\bar{J}(v_0) = f(v_0, s_0)$ and $(v_0/s_0, 1/s_0) = 
\nabla \bL( \eta_0, \l(\eta_0))\in
\mathcal{D}^\circ_{\bar{I}}$, Lemma \ref{Jderivative} implies that 
\[
 \nabla \bar{J}(v_0) = \left( \frac{\del \bar{I}}{\del x_i} 
 \left( \frac{v_0}{s_0}, \frac{1}{s_0}  \right)\right)_{i=1}^d.
\]
However, since $\nabla \bar{I}( v_0/s_0, 1/s_0) = 
\nabla \bar{I}( \nabla \bL( \eta_0, \l(\eta_0))) = 
(\eta_0, \l(\eta_0))$, we obtain that $\nabla \bar{J}(v_0) = \eta_0$. 
\end{proof}
\end{subsection}

\begin{subsection}{Properties of $\bar{J}$ - Nestling Case} \label{nestbJprop}

In this subsection, we will assume that $P$ is nestling, $d\geq 2$, and Assumptions \ref{asmUE}, \ref{asmIID}, and \ref{asmT} hold. 

\begin{lem}\label{bLnest}
 If $P$ is nestling, then $\bL(\eta, \l) = \infty$ for any $\l>0$. 
\end{lem}
\begin{proof}
Sznitman has shown \cite[Theorem 2.7]{sSlowdown} that when 
Assumptions \ref{asmUE}, \ref{asmIID}, and \ref{asmT} hold and $P$ 
is nestling and not marginally nestling, then 
\begin{equation}
 \liminf_{n\ra\infty}  
 \frac{\log \P(\tau_1 > n)}{(\log n)^d}> -\infty. \label{tautail}
\end{equation}
Sznitman proves \eqref{tautail} by constructing a ``trap'' of radius 
$\log n$ around the origin and then forcing the random walk to stay 
in the trap for at least the first $n$ steps of the walk. If instead we 
construct the trap centered around a point near $(\log n)\ell$, then we 
can adapt Sznitman's argument (using Assumption \ref{asmUE}) to show that 
when $P$ is nestling but not marginally nestling,
\begin{equation}
 \liminf_{n\ra\infty} \frac{\log \bP(\tau_1 > n)}
 {(\log n)^d }> \liminf_{n\ra\infty} 
 \frac{ \log \bP(\tau_1 > n, \; \| X_{\tau_1} \| < 
 3 \log n )}{(\log n)^d }> -\infty. \label{tautail2}
\end{equation}
In the marginally nestling case, we get immediately by approximating
a marginally nestling walk by a nestling walk for the first $n$ step
(at exponential cost $e^{-\e n}$),
that for any $\e>0$, 
\begin{equation}
  \liminf_{n\ra\infty} 
 \frac{ \log \bP(\tau_1 > n, \; \| X_{\tau_1} \| < 
 3 \log n )}{n}> -\e. \label{tautail2a}
\end{equation}
The statement of the lemma follows easily from \eqref{tautail2} and \eqref{tautail2a}.  
\end{proof}

For any $\eta\in \R^d$, let 
\begin{equation}
	\label{eq-061208b}
 \bL_X(\eta) = \bL(\eta, 0) = \log \bE e^{\eta\cdot X_{\tau_1}}.
\end{equation}
Recall the constant $C_1$ in 
Assumption 
\ref{asmT}, and
define the following subsets of $\R^d$:
\[
\mathcal{C}=\{\eta\in \R^d: \|\eta\|<C_1\},\qquad
 \mathcal{C}^+ = \mathcal{C} \cap \{ \bL_X(\eta) > 0 \} 
 \qquad \text{and} \qquad
 \mathcal{C}^0 = \mathcal{C} \cap \{ \bL_X(\eta) = 0 \}.
\]
As in the non-nestling case, for any $\eta\in \mathcal{C}^+ 
\cup \mathcal{C}^0$, let $\l(\eta)$ be the unique solution 
to $\bL(\eta, \l(\eta) ) = 0$. (Lemma \ref{bLnest} implies that $\bL(\eta, \l(\eta)) = 0$ does not have a solution when $\eta \in \mathcal{C} \backslash (\mathcal{C}^+ \cup \mathcal{C}^0)$ ). 
Note that $\l(\eta)$ is 
analytic and strictly concave on $\mathcal{C}^+$, 
and that $\l(\eta)= 0$ for all $\eta\in \mathcal{C}^0$. 
Define
\[
 \gamma(\eta) := \frac{ \bE X_{\tau_1} e^{\eta \cdot X_{\tau_1} +
 \l(\eta) \tau_1} }{ \bE \tau_1 e^{\eta \cdot X_{\tau_1} + \l(\eta) \tau_1} },
\]
so that $\gamma(\eta) = - \nabla \l(\eta)$ for $\eta \in \mathcal{C}^+$, 
and $\gamma(\eta)$ is continuous as a function of $\eta$. Also,
since $\l(\eta)$ is strictly concave as a function of 
$\eta$ in $\mathcal{C}^+$, then $\gamma(\eta)$ must be a one-to-one function. 
Let 
\[
 \mathcal{A}^+ := \gamma(\mathcal{C}^+) = \{ \gamma(\eta): 
 \eta\in \mathcal{C}^+ \}, \qquad \text{and} \qquad 
\mathcal{A}^0 := \gamma(\mathcal{C}^0) = \{ \gamma(\eta): 
\eta\in \mathcal{C}^0 \}.
\]
Then $\mathcal{A}^+$ is an open subset, and since 
$\vp = \bE X_{\tau_1}/ \bE \tau_1 = \gamma(\mathbf{0})$,
then $\vp \in \mathcal{A}^0 \subset \del \mathcal{A}^+$. 
Also, since $\mathcal{C}^0$ is a $d-1$ dimensional set with non-empty (relative) interior (by the implicit function theorem), the same is true of $\mathcal{A}^0$.



\begin{lem} \label{bJnest}
 Let $P$ be  nestling.
Then, $\bar{J}$ is analytic and 
strictly convex on the open set $\mathcal{A}^+$. Moreover, if 
\begin{equation}\label{vsnest}
 v_0 = \gamma(\eta_0) = \frac{ \bE X_{\tau_1} e^{\eta_0 \cdot X_{\tau_1} + \l(\eta_0) \tau_1} }{ \bE \tau_1 e^{\eta_0 \cdot X_{\tau_1} + \l(\eta_0) \tau_1} }, \qquad \text{and} \qquad
s_0 = \frac{1}{\bE \tau_1 e^{\eta_0 \cdot X_{\tau_1} + \l(\eta_0) \tau_1}}
\end{equation}
for some $\eta_0 \in \mathcal{C}^+$, then 
\[
 \bar{J}(v_0) = s_0 \bar{I}\left( \frac{v_0}{s_0}, \frac{1}{s_0} \right), \qquad \text{and} \qquad
\nabla{J}(v_0) = \eta_0,
\]
and $s_0$ is the unique value of $s$ which attains the minimum in the definition of $\bar{J}(v_0)$.
\end{lem}
\begin{proof}
The proof is exactly the same as the proof of Lemma \ref{nnanalytic}, and follows from the fact that $\bL(\eta, \l(\eta)) = 0$ for $\eta \in \mathcal{C}^+$ and the fact that $\bL(\eta, \l)$ is analytic and strictly convex in a neighborhood of $(\eta_0, \l(\eta_0))$ for any $\eta_0 \in \mathcal{C}^+$. 
\end{proof}

Since the sequence
$X_{\tau_1}, X_{\tau_2}-X_{\tau_1}, \ldots$ is \iid 
under $\bP$, Cram\'er's Theorem \cite[Theorem 6.1.3]{dzLDTA} 
implies that $X_{\tau_k}/k$ satisfies a large deviation principle 
under the measure $\bP$ with rate function $\bar{I}_1(x)$ given by
\[
 \bar{I}_1(x) = \sup_{\eta\in\R^d} \left[\eta \cdot x - \bL_X(\eta)\right].
\]
\begin{lem}\label{I1I2lb}
$\bar{I}_1(x) \leq \inf_{t\in\R} \bar{I}(x,t)$.
\end{lem}
\begin{proof}
The large deviation lower bound
\eqref{Thesis_XtauLDPlb} for $(X_{\tau_n}/n, \tau_n/n)$ implies that 
\[
\liminf_{n\ra\infty} \frac{1}{n} \log \bP( \| X_{\tau_n} - \xi n \| <
\d n) \geq -\inf_{\|x-\xi\| < \d, t\in\R} \bar{I}(x,t).
\]
On the other hand, the large deviation upper bound 
for $X_{\tau_n}/n$ implies that
\[
\limsup_{n\ra\infty} \frac{1}{n} \log \bP( \| X_{\tau_n} - \xi n \| 
< \d n) \leq -\inf_{\|x-\xi\| \leq \d} \bar{I}_1(x).
\]
The above two inequalities and the lower semicontinuity of
$\bar I$ and $\bar I_1$ imply that $\bar{I}_1(x)
\leq \inf_{t\in\R} \bar{I}(x,t)$.
\end{proof}

As mentioned above, when $d\geq 2$,
Assumptions \ref{asmUE}, \ref{asmIID}, and \ref{asmT} imply that 
$\bE \tau_1^p < \infty$ for all $p < \infty$. Then, for any 
$\eta \in \mathcal{C}$, by choosing $p$ large enough so 
that $\| \eta \| < \frac{p-1}{p} C_1$ we have that
\[
 \bE \tau_1 e^{\eta \cdot X_{\tau_1}} \leq \left( \bE \tau_1^p \right)^{1/p}
 \left( \bE e^{p/(p-1) \eta\cdot X_{\tau_1}} \right)^{(p-1)/p} < \infty.
\]
Then, for $\eta\in \mathcal{C}$, let $h(\eta) := 
\frac{\bE \tau_1 e^{\eta \cdot X_{\tau_1}}}{\bE 
e^{\eta \cdot X_{\tau_1}}}$, so that $\nabla \bL(\eta, 0) = 
\left( \nabla \bL_X(\eta), h(\eta) \right)$ (where the
derivatives with respect to $\l$ are one sided derivatives as $\l \ra 0^-$). 
\begin{lem} \label{Iconstant}
If $x = \nabla \bL_X(\eta)$ for some $\eta \in \mathcal{C}$, then $\bar{I}(x,t) = \bar{I}_1(x)$ for all $t \geq h(\eta)$. 
\end{lem}
\begin{proof}
	Since $\nabla\bL( \eta, 0) = (x,h(\eta))$, we have 
	using \eqref{eq-061208b} that 
\begin{align*}
 \bar{I}(x,h(\eta) ) = (x,h(\eta)) \cdot (\eta, 0) - 
 \bL(\eta, 0) = x \cdot \eta - \bL_X(\eta).
\end{align*}
Similarly, $\nabla \bL_X(\eta) = x$ implies that $\bar{I_1}(x) =
x \cdot \eta - \bL_X(\eta)$. Thus, $\bar{I}(x,h(\eta) ) = \bar{I}_1(x)$. 

If $t> h(\eta)$, then since Lemma \ref{bLnest} implies that $\bL(\eta, \l) = \infty$ for any $\l > 0$,
\begin{align*}
 \bar{I}(x,t) &= \sup_{\eta\in\R^d, \; \l \leq 0} (x,t)\cdot (\eta,\l) 
 - \bL(\eta,\l) \\
&\leq \sup_{\eta\in\R^d, \; \l \leq 0} (x,h(\eta))\cdot (\eta,\l)
- \bL(\eta,\l) \\
&= \bar{I}(x, h(\eta)) = \bar{I}_1(x). 
\end{align*}
This, along with Lemma \ref{I1I2lb} implies that $\bar{I}(x,t) = 
\bar{I}_1(x)$ for all $t\geq h(\eta)$. 
\end{proof}


Let $\mathcal{A}^- := \{ \theta v : \theta \in(0,1), \; v\in \mathcal{A}^0 \}$. 
In \cite{yPaper} (proof of Theorem 3, bottom of page 7), 
Yilmaz shows that the unit vector $\hat{n}$ normal to $\del \mathcal{A}^+$ 
(pointing into $\mathcal{A}^+$) at $\vp$ satisfies $\hat{n} \cdot \vp >0$. 
In fact, this argument gives that for any $v_0 \in \mathcal{A}^0$ the
unit vector $\hat{n}_0$ normal to $\del \mathcal{A}^+$ (pointing into 
$\mathcal{A}^+$) at $v_0$ satisfies $\hat{n}_0 \cdot v_0  >0$. This 
implies that $\mathcal{A}^-$ is an open set and that $\mathcal{A}^-$ 
and $\mathcal{A}^0$ are disjoint. 

\textbf{Remark:}
The above referenced argument of Yilmaz on the shape of $\mathcal{A}^+$ 
appears in a different form in \cite{yThesis} than it does here. 
Yilmaz defines a function $\Lambda_a(\eta)$ to be the Legendre transform
of the large deviation rate function $H(v)$. He then shows that 
the equality $\bL(\eta, -\Lambda_a(\eta)) = 0 $ holds for all 
$\eta \in \mathcal{C}^+$. 
Note that our definition of $\l(\eta)$ implies that $\L_a(\eta) = - 
\l(\eta)$ for all $\eta \in \mathcal{C}^+$, and thus 
\[
 \mathcal{A}^+ = \{ -\nabla \l(\eta): \eta \in\mathcal{C}^+ \} = 
 \{ \nabla \L_a(\eta): \eta \in\mathcal{C}^+ \}. 
\]
Since Yilmaz's proof of the properties of the normal vectors at points 
in $\mathcal{A}^0$ only uses the fact that $\bL(\eta, -\L_a(\eta)) = 0$, 
it may be repeated here with $-\l(\eta)$ in place of $\L_a(\eta)$.

We wish to identify the shape of the function $\bar{J}$ on the set $\mathcal{A}^-$ as well. For this we first need the following lemma.
\begin{lem}\label{J1bar}
 Let $\bar{J}_1(v) := \inf_{s>0} s \bar{I}_1(v/s)$. Then $\bar{J}_1(v)
 \leq \bar{J}(v)$ for all $v$, and $\bar{J}_1(c v) = c \bar{J}_1(v)$ for all $c>0$. 
Moreover, if $v_0 = \gamma(\eta_0)$ for some $\eta_0 \in \mathcal{C}^0$ and 
$c>0$, then $\bar{J}_1(v)$ is analytic in a neighborhood of $cv_0$.
\end{lem}
\begin{proof}
 Since $\bar{I}_1(x) \leq \inf_t \bar{I}(x,t)$, it follows immediately from the definitions of $\bar{J}$ and $\bar{J}_1$ that $\bar{J}_1(v) \leq \bar{J}(v)$. Also, if $c>0$, then 
\[
 \bar{J}_1(cv) = \inf_{s>0} s \bar{I}_1\left( \frac{cv}{s} \right) = c \inf_{s>0} (s/c) \bar{I}_1\left( \frac{v}{s/c} \right) = c \inf_{s'>0} s' \bar{I}_1\left( \frac{v}{s'} \right) = c \bar{J}_1(v). 
\]
Let $f_1(v,s) := s \bar{I}_1(v/s)$, so that $\bar{J}_1(v) = 
\inf_{s>0} f_1(v,s)$. Since $\bar{I}_1$ is a convex function,
$f_1(v,s)$ is a convex function of $(v,s)$. 
Let $v_0 \in \mathcal{A}^0$ so that $v_0 = \gamma(\eta_0) = \frac{ 
\nabla \bL_X(\eta_0)}{h(\eta_0)}$  for some $\eta_0 \in \mathcal{C}^0$. 
As in the proof of Lemma \ref{nnanalytic}, to show that $\bar{J}_1$ is
analytic in a neighborhood of $v_0$, by the implicit function theorem it 
is enough to show that there exists an $s_0$ such that $f(v,s)$ is 
analytic in a neighborhood of $(v_0,s_0)$, $\frac{\del f_1}{\del s}
(v_0, s_0) = 0$, and $\frac{\del^2 f_1}{\del s^2}(v_0,s_0) \neq 0$. 
If $s_0 = \frac{1}{h(\eta_0)}$, then $v_0/s_0 = \nabla \bL_X(\eta_0)$.
Since $\bL_X$ is analytic and strictly convex in a neighborhood of $\eta_0$,
it follows that $\bar{I}_1(x)$ is analytic and strictly convex in a
neighborhood of $v_0/s_0 = \nabla \bL_X(\eta_0)$
(see Lemma \ref{Appendix_analytic} in the Appendix). 
Thus, $f_1(v,s)$ is analytic in a neighborhood of $(v_0, s_0)$. 
When $\bar{I}_1$ is twice differentiable at $v/s$, then 
\begin{equation}
 \frac{\del f_1}{\del s} (v,s) = \bar{I}_1\left( \frac{v}{s} \right) - 
 \nabla \bar{I}_1\left( \frac{v}{s} \right) \cdot \left( \frac{v}{s} \right),  \label{df1ds}
\end{equation}
and
\begin{equation}
\frac{\del^2 f_1}{\del s^2} (v,s) = \frac{1}{s^3} v \cdot D^2 \bar{I}_1
\left( \frac{v}{s} \right) \cdot v^t. \label{PosDefI}
\end{equation}
Since $\nabla \bL_X(\eta_0) = v_0/s_0 $, we obtain that $\bar{I}_1(v_0/s_0)
= \eta_0 \cdot (v_0/s_0) - \bL_X(\eta_0)$ and thus \eqref{df1ds} implies
\begin{equation}
 \frac{\del f_1}{\del s} (v_0, s_0) =  \bar{I}_1\left( \frac{v_0}{s_0}
 \right) - \eta_0 \cdot \left( \frac{v_0}{s_0} \right) = -\bL_X(\eta_0) = 0, \label{df1zero}
\end{equation}
where the last equality is because $\eta_0 \in \mathcal{C}^0$. Also, since 
$\bar{I}(x)$ is strictly convex in a neighborhood of $v_0/s_0$,
$D^2\bar{I}_1(v_0/s_0)$ is strictly positive definite, and 
thus \eqref{PosDefI} implies that $\frac{\del^2 f_1}{\del s^2} (v_0, s_0) > 0$. 
Therefore, $\bar{J}_1$ is analytic in a neighborhood of $v_0$. Since
$\bar{J}_1(c v ) = c \bar{J}_1(v)$ for all $c>0$, this implies that
$\bar{J}_1$ is also analytic in a neighborhood of $c v_0$ for any $c>0$.
\end{proof}

\begin{lem}\label{bJnest2}
 Let $P$ be  nestling. Then, $\bar{J}(v) = \bar{J}_1(v)$ for all 
 $v \in \mathcal{A}^0 \cup \mathcal{A}^-$, and so
 $\bar{J}(v)$ is analytic and $1$-homogeneous on the open set $\mathcal{A}^-$. 
Moreover, if 
for some $\eta_0 \in \mathcal{C}^0$, 
\begin{equation}\label{vsnest2}
 v_0 = \gamma(\eta_0) = \frac{ \bE X_{\tau_1} e^{\eta_0 \cdot X_{\tau_1} } }{
 \bE \tau_1 e^{\eta_0 \cdot X_{\tau_1} } } \qquad \text{and} \qquad
s_0 = \frac{1}{h(\eta_0)} = \frac{1}{\bE \tau_1 e^{\eta_0 \cdot X_{\tau_1} }},
\end{equation}
then for any $\theta \in (0,1]$, 
\[
 \bar{J}(\theta v_0) = \theta s_0 \bar{I}_1\left( \frac{v_0}{s_0}\right) =
 \theta s_0 \bar{I} \left( \frac{v_0}{s_0}, \frac{1}{\theta s_0} \right), 
 \qquad  \qquad \nabla \bar{J}( \theta v_0) = \eta_0,
\]
and $\theta s_0$ is the unique value of $s$ which attains the minimum in
the definition of $\bar{J}(\theta v_0)$.
\end{lem}
\begin{proof}
Let $v_0$ and $s_0$ be defined as in \eqref{vsnest2} for some 
$\eta_0\in \mathcal{C}^0$. Recalling that $f_1(v,s)=s\bar{I}_1(v/s)$, then
\[
 \frac{\del f_1}{\del s}( \theta v_0, \theta s_0 )
 =\frac{\del f_1}{\del s}( v_0, s_0 ) =0,
\]
where the first equality holds because $\frac{\del f_1}{\del s}(v,s)$ 
depends only on $v/s$ by \eqref{df1ds}, and the second equality 
follows from
\eqref{df1zero}.
Therefore, $\bar{J}_1(\theta v_0) = f_1(\theta v_0, \theta s_0) = 
\theta s_0 \bar{I}_1(v_0/s_0)$. 
However, since $v_0/s_0 = \nabla \bL_X(\eta_0)$ and  
$ h(\eta_0) = 1/s_0 \leq 1/(\theta s_0)$ for any $\theta \in(0,1]$, 
we have by Lemma \ref{Iconstant} that 
\[
\theta s_0 \bar{I}_1\left( \frac{v_0}{s_0} \right) = \theta s_0
\bar{I}\left( \frac{v_0}{s_0}, \frac{1}{\theta s_0} \right) =
\theta s_0 \bar{I}\left( \frac{\theta v_0}{\theta s_0}, \frac{1}{\theta s_0}
\right) . 
\]
Thus, $\bar{J}_1(\theta v_0) \geq \bar{J}(\theta v_0)$. 
Since $\bar{J}_1(v) \leq \bar{J}(v)$ for all $v$, this implies
that $\bar{J}(v) = \bar{J}_1(v)$ for all
$v \in \mathcal{A}^0 \cup \mathcal{A}^-$. 
As in the proof of Lemma \ref{nnanalytic}, since $(v_0/s_0, 1/s_0) = 
\nabla \bL(\eta_0, 0)$ is in the interior of $\mathcal{D}_{\bar{I}}$, we can 
apply Lemma \ref{Jderivative} to show that $\nabla \bar{J}(v_0) = \eta_0$. 
Since $\bar{J}(\theta v_0) = \theta \bar{J}(v_0)$ for all $\theta \in(0,1]$ 
this implies that $\nabla \bar{J}(\theta v_0) = \eta_0$ as well. 

Since $v_0/s_0 = \nabla \bL_X(\eta_0)$, $\bar{I}_1$ is strictly convex 
in a neighborhood of $v_0/s_0$, and thus \eqref{PosDefI} 
implies that $f_1(v,s)$ is strictly convex in $s$ in a neighborhood of 
$(\theta v_0, \theta s_0)$. Therefore, $\theta s_0$ is the unique minimizing value of $s$ 
in the definition of $\bar{J}_1(\theta v_0)$. Since $f_1(v,s) \leq f(v,s)$, 
this implies that $\theta s_0$ is the unique minimizing value of $s$ 
in the definition of $\bar{J}(\theta v_0)$ as well. 
\end{proof}

\begin{cor}\label{Jzeroset}
 If $P$ is nestling, then $\bar{J}( \theta \vp ) = 0$
 for all $\theta\in (0,1]$.
\end{cor}
\begin{proof}
Since $\vp = \gamma(\mathbf{0}) \in \mathcal{A}^0$, 
Lemma \ref{bJnest2} implies that $\bar{J}(\theta \vp) = \theta \bar{J}(\vp)$. 
However, $\bar{J}(\vp) = 0$ by Lemma \ref{bJsimleprop}. 
\end{proof}

\begin{cor} \label{bJnest3}
 If $P$ is nestling, then $\bar{J}(v)$ is continuously differentiable 
 on the open set $\mathcal{A}:= \mathcal{A}^- 
 \cup \mathcal{A}^0 \cup\mathcal{A}^+$, 
 and $\| \nabla \bar{J}(v) \| < C_1$ for all $v \in \mathcal{A}$.
\end{cor}
\begin{proof}
This is a direct application of the formulas given for $\nabla\bar{J}(v)$ in 
Lemmas \ref{bJnest} and \ref{bJnest2} and the fact that $\gamma(\eta)$ 
is continuous and one-to-one on $\mathcal{C}^+ \cup \mathcal{C}^0$.  
\end{proof}
\end{subsection}
\end{section}

\begin{section}{LDP Lower Bound}\label{LDPlowerbound}
We now prove, in both the nestling and non-nestling cases,
the large deviation lower bound.
\begin{prop}[Lower Bound] \label{Thesis_ldplb}
Let Assumptions \ref{asmUE}, \ref{asmIID}, and \ref{asmT} hold. For any $v \in H_\ell$, 
\[
\lim_{\d\ra 0} \liminf_{n\ra\infty} \frac{1}{n} \log \P( \| X_n - nv \| < n \d ) \geq - \bar{J}(v).
\]
\end{prop}
\begin{proof}
Let $\| \xi \|_1$ denote the $L^1$ norm of the vector $\xi$. 
Then, it is enough to prove the statement of the proposition with 
$\| \cdot \|_1$ in place of $\| \cdot \|$. 
Also, since $\P( \| X_n - nv \|_1 < n \d ) 
\geq \P(D) \bP( \| X_n - nv \|_1 < n \d )$, 
it is enough to prove the statement of the proposition with $\bP$ 
in place of $\P$. That is, it is enough to show
\[
\lim_{\d\ra 0} \liminf_{n\ra\infty} \frac{1}{n} \log 
\bP( \| X_n - nv \|_1 < n \d ) \geq - \bar{J}(v).
\] 

Now, for any $\d>0$ and any integer $k$, since the walk is a nearest neighbor 
walk,
\[
\bP( \| X_n - nv \|_1 < 4 n\d ) \geq \bP\left( \| X_{\tau_k} - 
nv \|_1 < 2 n \d, \; | \tau_k - n | < 2n\d \right). 
\]
For any $t\geq 1$, let $k_n=k_n(t):= \lfloor n/t \rfloor$, 
so that $n - t < k_n t \leq n$ for all $n$. 
Thus, for any $\d>0$ 
and $t\geq 1$, and for all $n$ large enough (so that $n\d > t$), 
\begin{align*}
\bP( \| X_n - nv \|_1 < 4 n\d ) &\geq \bP\left( \| X_{\tau_{k_n}} - 
nv \|_1 < 2 n \d, \; | \tau_{k_n} - n | < 2n\d \right)
\\
&\geq \bP\left( \| X_{\tau_{k_n}} - k_n t v\|_1 <
k_n t \d, \; | \tau_{k_n} - k_n t | < k_n t \d \right).
\end{align*}
Therefore, for any $\d>0$ and $t\geq 1$,
\begin{align*}
&\liminf_{n\ra\infty} \frac{1}{n} \log \bP( \| X_n - nv \|_1 < 4 n\d )  \\
&\qquad \geq \liminf_{n\ra\infty} \frac{1}{n} \log \bP\left( \|
X_{\tau_{k_n}} - k_n t v\|_1 <
k_n t \d, \; | \tau_{k_n} - k_n t | < k_n t \d \right) \\
&\qquad \geq \frac{1}{t} \liminf_{n\ra\infty} \frac{1}{k_n} \log
\bP\left( \| X_{\tau_{k_n}} - k_n t v\|_1 <  k_n t \d, \; | \tau_{k_n}
- k_n t | < k_n t \d \right) \\
&\qquad = \frac{1}{t} \liminf_{k\ra\infty} \frac{1}{k} \log
\bP\left( \| X_{\tau_{k}} - k t v\|_1 <  k t \d, \; |
\tau_{k} - k t | < k t \d \right)\\
&\qquad = -\frac{1}{t} \inf_{ \substack{  \| x - tv\|_1 <
t\d \\  | y - t | < t \d } } \bar{I}(x,y),
\end{align*}
where the last equality is from \eqref{Thesis_XtauLDPlb}. 
Taking $\d\ra 0$ we get that for any $t\geq 1$, 
\[
\lim_{\d\ra 0} \liminf_{n\ra\infty} \frac{1}{n} \log
\bP( \| X_n - nv \|_1 < 4 n\d ) \geq -\frac{1}{t} \bar{I}(vt,t). 
\]
Since the last inequality holds
for any $t$, the proof is completed by taking the supremum of 
the right side over all $t\geq 1$ and 
recalling the definition of $\bar{J}$.  
\end{proof}
\end{section}

\begin{section}{LDP Upper Bound}\label{Thesis_upperbound}
We now wish to prove a matching large deviation 
upper bound to Proposition \ref{Thesis_ldplb}, still working 
under Assumptions \ref{asmUE}, \ref{asmIID} and \ref{asmT}. 
Ideally, we would like for the upper bound to be valid for
all $v \in H_\ell$. 
This is possible for $d=1$ (see the remarks at the end of the paper), 
but for 
$d>1$ we are only able to prove a matching upper bound to Proposition
\ref{Thesis_ldplb} in a neighborhood of the set where $\bar{J}(v)$ 
equals zero. 
However, this is enough to be able to prove Theorems 
\ref{Thesis_LDPdiff} and \ref{nestLDPdiff}.

A key step in proving the large deviation upper bound in both the 
non-nestling and nestling cases is the following upper bound 
involving regeneration times:
\begin{lem}\label{chcvlem}
For any $t,k\in \N$ and any $x\in \Z^d$,
\[
\bP( X_{\tau_k} = x, \; \tau_k = t ) 
\leq e^{-t \bar{J}\left( \frac{x}{t} \right) } . 
\]
\end{lem}
\begin{proof}
Chebychev's inequality implies that, for any $\l \in \R^{d+1}$,
\[
\bP\left( X_{\tau_k}=x, \tau_k=t \right) \leq e^{-\l\cdot(x,t)} 
\bE e^{\l \cdot(X_{\tau_k},\tau_k)} = e^{-k\left( \l\cdot(x/k, t/k) - 
\bL(\l) \right)},
\]
where in the last equality we used the \iid structure of 
regeneration times from Theorem \ref{Thesis_regindep}. 
Thus, taking the infimum over all $\l \in \R^{d+1}$ and 
using the definition of $\bar{J}$ (with $s=\frac{k}{t}$),  
\[
\bP\left( X_{\tau_k}=x, \tau_k=t \right) 
\leq e^{-k \bar{I}\left( \frac{x}{k}, \frac{t}{k} \right) } = e^{-t 
\frac{k}{t} \bar{I}\left( \frac{x}{t} \frac{t}{k}, 
\frac{t}{k} \right) } \leq e^{-t \bar{J}\left( \frac{x}{t} \right)} .
\]
\end{proof}

\begin{subsection}{LDP Upper Bound - Non-nestling Case} \label{nnldpub}
We are now ready to give a matching upper bound to
Proposition \ref{Thesis_ldplb} in a neighborhood of $\vp$.
Let $\mathcal{A}' := \{ v \in \R^d: \| \nabla \bar{J}(v) \| < 
\frac{C_2}{4} \}$. Note that Lemma \ref{nnanalytic} 
implies that $\mathcal{A}' \subset \mathcal{A}$. 
\begin{prop}[Upper Bound] \label{Thesis_ldpub}
Let Assumptions \ref{asmUE}, \ref{asmIID}, and \ref{asmT} hold, 
and let $P$ be non-nestling in direction $\ell$. 
Then, if $v\in \mathcal{A}'$ and $\d>0$ is sufficiently small, 
\[
\limsup_{n\ra\infty} \frac{1}{n} \log \P \left( \|X_n - n v \| < n
\d \right) \leq - \inf_{\|x - v\| < \d}  \bar{J}(x).
\]
\end{prop}
\begin{proof}
Since $\bar{J}$ is convex, $0 = \bar{J}(\vp) \geq \bar{J}(v) + 
\nabla \bar{J}(v) \cdot( \vp - v)$. Then, since $\| 
\nabla \bar{J}(v) \| < C_2/4$ for any $v\in \mathcal{A}'$ 
we have that $\bar{J}(v) \leq C_2/4 \| \vp - v \| < C_2/2$. 
Thus, for a fixed $v \in \mathcal{A}'$ 
we can choose a $\d>0$ and an $\e \in (0,1/2)$ 
such that $\bar{J}(v') < \e C_2$ 
and $\nabla \bar{J}(v') < C_2/4$ for all $\| v' - v\| < \d$.

Recalling \eqref{tau1exptail}, we obtain that 
there exist constants $C_3, C_2> 0$ such that 
\[
 \max \left\{ \P( \tau_1 > t ) , \bP( \tau_1 > t ) \right\} 
 \leq C_3 e^{-C_2 t}, \qquad \forall t>0.
\]
Let $v\in \mathcal{A}'$, and let $\e,\d>0$ be chosen as above. Now,
\begin{align}
\P( \| X_n - n v \| < n \d ) &\leq \P(\exists k
\leq n: \tau_k-\tau_{k-1} \geq \e n)  \nonumber \\
&\qquad + \P\left( \exists k: \tau_1 < \e n, \; \tau_k 
\in (n-\e n, n], \; \| X_n - n v \| < n \d, \; \tau_{k+1} > n 
\right). \label{Thesis_tausmall}
\end{align}
Then, since $\bar{J}(v) < \e C_2 $, 
\[
\P(\exists k\leq n: \tau_k-\tau_{k-1} \geq \e n) \leq C_3 n e^{-C_2 \e n} 
\leq C_3 n e^{-n\bar{J}(v)}.
\]
Thus, we need only to bound the second term in \eqref{Thesis_tausmall}.  

Since the random walk is a nearest neighbor walk, $\| X_{\tau_k} - nv \| 
\leq \| X_n - nv \| + | n-\tau_k|$. Thus,
\begin{align}
&\P\left( \exists k: \tau_1 < \e n, \; \tau_k \in (n-\e n, n], \; \| 
X_n - n v \| < n \d, \; \tau_{k+1} > n \right) \nonumber \\
&\qquad \leq  \sum_{k\leq n} \sum_{u\in(0,\e)} \sum_{s\in [0,\e )} 
\P\left( \tau_1 = u n, \; \tau_k = (1-s)n, \; \|
X_{\tau_k} - n v \| < n(\d + s), \; \tau_{k+1} > n \right),  \nonumber
\end{align}
where the above sums are only over the finite number of possible $u$ and $s$ 
such that the probabilities are non-zero. However, 
\begin{align*}
&\P\left( \tau_1 = u n, \; \tau_k = (1-s)n, \; \|
X_{\tau_k} - n v \| < n(\d + s), \; \tau_{k+1} > n  \right)\\
&\qquad \leq \P\left( \tau_1 = u n, \; \tau_k-\tau_1 = (1-s-u)n, \; \| 
X_{\tau_k} - X_{\tau_1} - nv \| \leq n (\d + s + u), \; \tau_{k+1}-\tau_k > 
n s \right) \\
&\qquad = \P(\tau_1 = u n) \bP\left( \tau_{k-1} = (1-s-u)n, \; \| 
X_{\tau_{k-1}} - nv \| \leq n(\d+s+u) \right) \bP(\tau_1 > ns),
\end{align*}
where the first inequality again uses the fact that the random walk 
is a nearest neighbor random walk, and the last equality 
uses the independence structure of regeneration times from Theorem 
\ref{Thesis_regindep}. Thus, since $\P(\tau_1 = u n ) 
\leq C_3 e^{-C_2 u n}$ and $\bP(\tau_1 > ns ) \leq C_3 e^{-C_2 s n}$,  
\begin{align}
&\P\left( \exists k: \tau_1 < \e n, \; \tau_k \in (n-\e n, n], \; \|
X_n - n v \| < n \d, \; \tau_{k+1} > n \right) \nonumber \\
& \quad  \leq \sum_{k\leq n} \sum_{u\in(0,\e)} \sum_{s\in [0,\e )}
C_3^2 e^{-C_2 (u+s) n} \bP\left( \tau_{k-1} = (1-s-u)n, \; 
\|X_{\tau_{k-1}} - n v \| < n(\d + s + u) \right). \label{Thesis_smallfirsttau}
\end{align}
By Lemma \ref{chcvlem}, the last expression is 
bounded above by
\begin{align}
&\sum_{k\leq n} \sum_{u\in(0,\e)} \sum_{s\in [0,\e )} \sum_{\| x - v \| <
\d + u+s} e^{-n(1-s-u)\bar{J}\left( \frac{x}{1-s-u} \right)} 
C_3^2 e^{-C_2 (s+u) n } \nonumber \\
&\qquad \leq C_4 n^{d+3} \sup_{s\in [0,2\e)} \sup_{\|x - v\| <
\d + s} e^{-n\left( (1-s)\bar{J}\left( \frac{x}{1-s} \right) + 
C_2 s \right)} \nonumber \\
&\qquad = C_4 n^{d+3} 
\exp \left\{ -n \left( \inf_{s\in[0,2\e)} \inf_{\|x - v\| < \d + s}
(1-s)\bar{J}\left( \frac{x}{1-s} \right) + 
C_2 s \right) \right\}, \label{Thesis_2sups}
\end{align}
for some constant $C_4$. 

To finish the proof of the proposition, it is enough to show that the
infimum in \eqref{Thesis_2sups} is achieved when $s=0$. 
That is, it is enough to show the infimum is 
larger than $\inf_{\|x-v\| < \d} \bar{J}(x)$. 
To this end, note first that 
\begin{align}
\inf_{s\in[0,2\e)} & \inf_{\|x - v\| < \d + s} 
(1-s)\bar{J}\left( \frac{x}{1-s} \right) + C_2 s = 
\inf_{\|x-v\| < \d} \inf_{s\in[0,2\e)} \inf_{\|y - x\| < s} 
(1-s)\bar{J}\left( \frac{y}{1-s} \right) + C_2 s.  \label{Thesis_3infs}
\end{align}
Since $\bar{J}$ is convex,
\[
 \bar{J}\left( \frac{y}{1-s} \right) \geq \bar{J}(x) + 
 \nabla \bar{J}(x) \cdot\left( \frac{y}{1-s} - x \right)
 \geq \bar{J}(x) - \|\nabla \bar{J}(x)  \| \left\|\frac{y}{1-s} - x \right \| .
\]
If $\|y-x \| < s$ and $\| \nabla \bar{J}(x) \| < C_2/4$ this implies that
\[
 (1-s) \bar{J}\left( \frac{y}{1-s} \right) + C_2 s
 \geq  (1-s) \bar{J}(x) - \frac{C_2}{2}s + C_2 s = \bar{J}(x) + 
 \left(\frac{C_2}{2} - \bar{J}(x)\right) s.
\]
Recalling \eqref{Thesis_3infs}, we obtain 
\begin{align*}
\inf_{s\in[0,2\e)} \inf_{\|x - v\| < \d + s} (1-s)\bar{J}\left( \frac{x}{1-s}
\right) + C_2 s  &\geq  \inf_{\|x-v\| < \d} \inf_{s\in[0,2\e)} \bar{J}(x) + 
s\left( \frac{C_2}{2} - \bar{J}(x) \right)\\
&= \inf_{\|x-v\| < \d} \bar{J}(x) ,
\end{align*}
where the last inequality is because our choice of $\d$ and $\| x - v\| < 
\d$ imply that $\bar{J}(x)  < \e C_2 < \frac{C_2}{2}$. 
This completes the proof of the proposition.
\end{proof}
\end{subsection}

\begin{subsection}{LDP Upper Bound - Nestling Case}\label{NestlingUpperbound}
Before proving a large deviation upper bound in 
the nestling case, we need the following lemma.
\begin{lem} \label{nestchcv}
Assume $P$ is nestling. If $xn\in \Z^d$ and $k\leq n$, then
\[
\bP(X_{\tau_k} = xn, \; \tau_k \leq n) \leq n e^{-n\bar{J}(x)}.
\]
\end{lem}
\begin{proof}
Lemma \ref{chcvlem} implies that
\begin{align*}
 \bP(X_{\tau_k} = xn,  \; \tau_k \leq n) &=\sum_{\theta\in(0,1], \; \theta n \in
 \Z } \bP(X_{\tau_k} = xn,  \; \tau_k = \theta n) \leq \sum_{\theta\in(0,1], \; 
 \theta n \in \Z  } e^{-n \theta \bar{J}\left( \frac{x}{
 \theta} \right)}. 
\end{align*}
Then, we will be finished if we can show that 
$\theta \bar{J}\left( \frac{x}{\theta} \right) \geq \bar{J}(x)$. 
The convexity of $\bar{J}$ implies that $\theta \bar{J}\left( 
\frac{x}{\theta} \right) \geq \bar{J}( x + (1-\theta) z ) - (1-\theta) 
\bar{J}(z)$ for any $z$. Letting $z=c \vp $ for some $c \in(0,1]$,
Lemma \ref{Jzeroset} implies that $\theta \bar{J}\left( \frac{x}{\theta} 
\right) \geq \bar{J}( x + (1-\theta) c \vp )$. 
Letting $c\ra 0^+$ completes the proof.
\end{proof}

We are now ready to prove a matching large deviation upper bound to 
Proposition \ref{Thesis_ldplb} in the nestling case. 
The proof is similar to the proof of the upper bound in the non-nestling case. 
However, instead of forcing regeneration times to be small, 
we instead force regeneration distances to be small. 
\begin{prop} \label{nestldpub}
Let Assumptions \ref{asmUE}, \ref{asmIID}, and \ref{asmT} hold, 
let $P$ be nestling, and let $d\geq 2$. 
Then, if $v\in \mathcal{A}$ and $\d>0$ is sufficiently small, 
\[
 \limsup_{n\ra\infty} \frac{1}{n} \log \P( \|X_n - n v\| < \d n ) \leq 
 -\inf_{\|x-v\|\leq  \d} \bar{J}(x). 
\]
\end{prop}
\begin{proof}
 Corollary \ref{bJnest3} implies that $\| 
 \nabla \bar{J}(v) \| <  C_1$ for all $v \in \mathcal{A}$. 
 Since $\bar{J}$ is convex, $\bar{J}(z) \geq \bar{J}(v) + 
 \nabla \bar{J}(v) \cdot(z-v)$ for any $z$. 
 Letting $z= \theta \vp$ for any $\theta \in(0,1]$, Lemma \ref{Jzeroset} 
 implies that $\bar{J}(v) \leq - \nabla \bar{J}(v) \cdot (\theta \vp - v) 
 \leq \| \nabla \bar{J}(v) \| \| \theta\vp -v \|$. Letting $\theta\ra 0^+$, 
 we obtain that $\bar{J}(v) \leq \| \nabla \bar{J}(v) \| \| v \| < 
 C_1 \| v \|$. 

Now, for a fixed $v \in \mathcal{A}$, choose a $\d>0$ and a $c< \|v\| - \d$  
such that $\bar{J}(v') < c C_1$ and $\| \nabla \bar{J}(v')\| < C_1$ 
for all $\|v' - v\| < \d$. 
Letting $\tau_0:= 0$, we define $S_k:= \sup_{\tau_k < n \leq \tau_{k+1}} \| 
X_n - X_{\tau_k} \|$. 
By Assumption \ref{asmT}, it is clear that there exists a constant 
$C_5>0$ such that
\begin{equation}\label{Stails}
 \max \left\{ \bP\left( S_0 > t \right), \P\left( S_0  > t \right) 
 \right\} = \max \left\{ \bP\left( \sup_{n<\tau_1} \|X_{n} \| > t 
 \right), \P\left( \sup_{n<\tau_1} \|X_{n} \| > t \right) 
 \right\} \leq C_5 e^{-C_1 t}
\end{equation}
Then,
\begin{align}
&\P( \|X_n - n v\| < \d n ) \nonumber \\
&\qquad \leq \P( S_0 \geq c n ) + n\bP( S_0 \geq c n) 
 + \P( \| X_n -n v\| < \d n, \; S_i < c n \quad \forall i=0,1,\ldots n)  
 \nonumber \\
&\qquad \leq C_5(n+1)e^{-C_1 c n} 
 + \P( \| X_n -n v\| < \d n, \; S_i < c n \quad \forall i=0,1,\ldots n )  
 \nonumber \\
&\qquad \leq C_5(n+1)e^{-n \bar{J}(v)} 
 + \P( \| X_n -n v\| < \d n, \; S_i < c n \quad \forall i=0,1,\ldots n ), 
 \label{largeregdist}
\end{align}
where the last inequality is because $\bar{J}(v) < c C_1$. Thus, it is 
enough to bound the second term on the right side of \eqref{largeregdist}.
Since $c < \| v \| - \d$, 
the event $\{ \|X_n - n v \| < \d n , \; S_0 < c n \}$ 
implies that $\tau_1 < n$. 
Decomposing according to the last regeneration time before $n$, 
we obtain that
\begin{align}
& \P( \| X_n -n v\| < \d n, \; S_i <  c  n \quad \forall i=0,1,\ldots n )  \nonumber \\
&\qquad = \sum_{k=1}^n \P(\tau_k \leq n < \tau_{k+1}, \; \| X_n -n v\| < \d n, \; S_i <  c  n \quad \forall i=0,1,\ldots n ) \nonumber \\
&\qquad \leq \sum_{k=1}^n \sum_{\|x\| <  c } \sum_{\|y\| <  c  }\sum_{\|z\| < \d} \P(X_{\tau_1} = x n, \;  X_{\tau_k} = n(v+z-y), \; X_n= n(v+z), \; \tau_k \leq n < \tau_{k+1} ), \label{foursums}
\end{align}
where the above sums are only over the finite number of possible $x, y,$ and $z$ such that the probabilities are non-zero. The \iid structure of regeneration times and distances from Theorem \ref{Thesis_regindep} implies that
\begin{align*}
& \P(X_{\tau_1} = x n, \;  X_{\tau_k} = n(v+z-y), \; X_n= n(v+z), \; \tau_k \leq n < \tau_{k+1} )\\
& \qquad \leq \P(X_{\tau_1} = x n) \bP(X_{\tau_{k-1}} = n(v+z-y-x), \; \tau_{k-1}\leq n) \bP( S_0 \geq \|y\|n) \nonumber \\
&\qquad \leq C_5 e^{-C_1 \|x\|n}  e^{-n\bar{J}\left( v+z-y-x \right)} C_5 e^{-C_1 \|y\|n},
\end{align*}
where in the last inequality we used \eqref{Stails} and Lemma \ref{nestchcv}.
Since there are at most $C_6 n^{3d+1}$ terms in the sum in \eqref{foursums} for some constant $C_6$ depending only on $c, \d$, and $d$, we obtain that
\begin{align}
& \P( \| X_n -n v\| < \d n, \; S_i <  c  n \quad \forall i=0,1,\ldots n ) \nonumber \\
&\qquad \leq C_6 n^{3d+1} \exp \left\{ -n \left(\inf_{\|z\| < \d}\inf_{\|x\| < c } \inf_{\|y\| < c} \bar{J}(v+z-x-y)+ C_1( \|x\| + \|y\|)  \right)   \right\}.  \label{nestinfs}
\end{align}
However, the convexity of $\bar{J}$ and the fact that $\| \nabla \bar{J}(v+z) \| < C_1$ for all $\|z\| < \d$ imply that
\[
 \bar{J}(v+z-x-y) \geq \bar{J}(v+z) + \nabla \bar{J}(v+z) \cdot (-x-y) \geq \bar{J}(v+z) - C_1( \|x\| + \|y\|). 
\]
Thus, the infimum in \eqref{nestinfs} is achieved when $\|x\| = \|y\| = 0$, and therefore,
\[
 \P( \| X_n -n v\| < \d n, \; S_i <  c  n \quad \forall i=0,1,\ldots n ) \leq C_6 n^{3d+1} \exp \left\{ -n \inf_{\|z\| < \d} \bar{J}(v+z)   \right\}.
\]
This, combined with \eqref{largeregdist} completes the proof of the proposition.
\end{proof}

\end{subsection}

Finally, we give the proofs of the main results of this paper.
\begin{proof}[\textbf{Proof of Theorems \ref{Thesis_LDPdiff} and \ref{nestLDPdiff}:}]$\left.\right.$\\
The annealed large deviation principle in Theorem \ref{annealedLDP} implies that 
\[\lim_{\d\ra 0} \liminf_{n\ra 0} \frac{1}{n} \log \P( \| X_n - nv \| < n \d ) = -H(v).\] 
Then, if the law on environments is non-nestling, Propositions \ref{Thesis_ldplb} and  \ref{Thesis_ldpub} imply that $\bar{J}(v) = H(v)$ for all $v \in \mathcal{A}'$ (where $\mathcal{A}'$ is defined as in the beginning of Subsection \ref{nnldpub}). 
Similarly, if $P$ is nestling, Propositions \ref{Thesis_ldplb} and  \ref{nestldpub} imply that $\bar{J}(v) = H(v)$ for all $v \in \mathcal{A}$ (where $\mathcal{A} = \mathcal{A}^+ \cup \mathcal{A}^0 \cup \mathcal{A}^-$ was defined as in Subsection \ref{nestbJprop}). The properties of $\bar{J}(v)$ given in Subsections \ref{nnbJprop} and \ref{nestbJprop} are then also true for $H(v)$.
\end{proof}

\end{section}

\begin{section}{Concluding Remarks and Open Problems}
 
\begin{enumerate}

\item The function $\bar{J}$ depends implicitly on the direction $\ell$ 
chosen for the definition of the regeneration times. 
Write $\bar{J}^\ell$ to make this dependence explicit. A consequence 
of our proofs of Theorems \ref{Thesis_LDPdiff} and \ref{nestLDPdiff} 
is that for any $\ell, \ell' \in \{ \xi \in S^{d-1}: \xi \cdot \vp > 0, \; c 
\xi \in \Z^d \text{ for some } c>0 \}$, $\bar{J}^\ell(v) = \bar{J}^{\ell'}(v)$ 
for all $v$ in some  neighborhood of where the functions are zero. 

\begin{op}\label{ldependence}
Recall that $\bar{J}^\ell$ is defined on $H_\ell = \{ v\in \R^d: v\cdot \ell > 0\}$. 
Is it true that $\bar{J}^\ell(v) = \bar{J}^{\ell'}(v)$ for all $v \in H_\ell \cap H_{\ell'}$?
\end{op}

\item The large deviations lower bound in Proposition \ref{Thesis_ldplb} holds for all $v \in H_\ell$, but we were only able to prove a matching upper bound in a neighborhood of the set where $\bar{J}(v) = 0$. 
However, if $d=1$ then we are able to prove a matching upper bound for all $v \in H_\ell$:
\begin{prop}[Proposition 6.3.11 in \cite{MyThesis}]\label{onedimub}
 Let $X_n$ be a RWRE on $\Z$. Let Assumptions \ref{asmUE} and \ref{asmIID} hold, and assume that $\P(\lim_{n\ra\infty} X_n = + \infty) = 1$. Define $\bar{J}$ as above in terms of regeneration times in the  direction $\ell = 1$. Then, for any $v > 0$ and $\d<v$,
\[
 \lim_{n\ra\infty} \frac{1}{n} \log \P( |X_n - nv | < n\d ) = - \inf_{x: |x-v| < \d} \bar{J}(x).
\]
\end{prop}
 
The following remains an open question:
\begin{op}\label{halfplaneconjecture}
 Do the large deviation upper bounds in Propositions \ref{Thesis_ldpub} and \ref{nestldpub} hold for all $v\in H_\ell$?
\end{op}
\textbf{Note:} An affirmative answer to Question \ref{halfplaneconjecture} would imply that $H(v) = \bar{J}(v)$ for all $v \in H_\ell$.
This would therefore imply that the answer to Question \ref{ldependence} is also affirmative. 
\end{enumerate}

\noindent\textbf{Acknowledgements:} We thank Atilla Yilmaz for pointing
out to us his work \cite{yThesis}.
\end{section}

\appendix


\begin{section}{Analyticity of Legendre Transforms}\label{Aanalytic}
Let $F:\R^d \ra \R$ be a convex function. Then, the Legendre transform $F^*$ of $F$ is defined by 
\begin{equation}
 F^*(x) = \sup_{\l \in\R^d} \l \cdot x - F(\l). \label{FLdef}
\end{equation}
\begin{lemA}\label{Appendix_analytic}
Let $F$ be strictly convex and analytic on an open subset $U\subset \R^d$. Then, $F^*$ is strictly convex and analytic in 
$U' := \{ y\in \R^d : y= \nabla F(\l) \text{ for some } \l \in U \}$. 
\end{lemA}
\begin{proof}
Since $F$ is strictly convex on $U$, $\nabla F$ is one-to-one on $U$. Therefore, for any $x\in U'$, there exists a unique $\l(x) \in U$ such that $\nabla F(\l(x)) = x$. (That is, $x\mapsto \l(x)$ is the inverse function of $\nabla F$ restricted to $U$.) This implies, 
since $\l \mapsto \l \cdot x - F(\l)$ is a concave function in $\l$, that the supremum in \eqref{FLdef} is achieved with $\l=\l(x)$ when $x\in U'$. 
That is, 
\begin{equation}
 F^*(x) = \l(x) \cdot x - F\left((\l(x)\right), \qquad \forall x\in U'. \label{maxachieved}
\end{equation}

Since $F$ is analytic on $U$, then $\nabla F$ is also analytic on $U$. Then, a version of the inverse function theorem \cite[Theorem 7.5]{fgHolomorphic} implies that $\l(\cdot)$ is analytic on $U'$ if 
\begin{equation}
 \det \left(D^2 F(x)\right) \neq 0 , \qquad \forall x \in U, \label{Jacobianneq0}
\end{equation}
where $D^2 F$ is the matrix of second derivatives of $F$. However, since $F$ is strictly convex on $U$, $D^2 F(x)$ is strictly positive definite for all $x\in U$. Thus, \eqref{Jacobianneq0} holds and so $x\mapsto \l(x)$ is analytic on $U'$. Recalling \eqref{maxachieved}, we then obtain that $F^*$ is also analytic on $U'$. 

An application of the chain rule to \eqref{maxachieved} implies that 
\[
 \nabla F^*(x) = \l(x) \quad\text{and}\quad D^2 F^* (x) = D \l(x) = \left( D^2 F  (\l(x)) \right)^{-1}, \qquad \forall x\in U'.
\]
Since $D^2 F$ is strictly positive definite on $U$, the above implies that $D^2 F^*(x)$ is strictly positive definite for all $x\in U'$. Thus $F^*$ is strictly convex on $U'$. 
\end{proof}

\end{section}


\providecommand{\bysame}{\leavevmode\hbox to3em{\hrulefill}\thinspace}
\providecommand{\MR}{\relax\ifhmode\unskip\space\fi MR }
\providecommand{\MRhref}[2]{%
  \href{http://www.ams.org/mathscinet-getitem?mr=#1}{#2}
}
\providecommand{\href}[2]{#2}

\end{document}